\documentclass[12pt,reqno]{amsart}
\usepackage{graphicx} 
\usepackage{amssymb}
\usepackage{amsthm}
\usepackage{enumitem}
\usepackage{xcolor}
\usepackage{stmaryrd}
\usepackage[margin = 1.15 in]{geometry}
\usepackage{todonotes}
\usepackage{stmaryrd}
\usepackage{subcaption}
\usepackage{caption}

\usepackage[
	colorlinks,
	linkcolor={black},
	citecolor={black},
	urlcolor={black}
]{hyperref}

\title{Dynamical Lifshitz Tails}
\author[\'I. Emilsd\'ottir]{\'Iris Emilsd\'ottir}
\thanks{\'I.\,E.\ and G.\,M.\ were supported in part by NSF grant DMS-2247966 (PI: A. Gorodetski).}
\address{Department of Mathematics, University of California, Irvine, CA 92697, USA}
\email{iris.e@uci.edu}

\author[G. Monakov]{Grigorii Monakov}
\address{Department of Mathematics, University of California, Irvine, CA 92697, USA}
\email{gmonakov@uci.edu}

\newcommand{\bbN}{{\mathbb{N}}}
\newcommand{\bbR}{{\mathbb{R}}}

\newcommand{\bbP}{{\mathbb{P}}}
\newcommand{\bbZ}{{\mathbb{Z}}}

\newcommand{\s}{{\mathrm{S}}}

\newcommand{\eps}{\varepsilon}

\newcommand{\Prob}{\mathbb{P}}

\newcommand{\bomega}{\bar{\omega}}

\DeclareMathOperator{\supp}{supp}

\newtheorem{theorem}{Theorem}[section]
\newtheorem{prop}[theorem]{Proposition}
\newtheorem{lemma}[theorem]{Lemma}
\theoremstyle{definition}
\newtheorem{remark}[theorem]{Remark}
\newtheorem{coro}[theorem]{Corollary}

\newtheorem*{claim*}{Claim}

\begin{document}

\begin{abstract}
We consider one-parameter families of random circle diffeomorphisms $g_{E,y}$ for which the unperturbed map $g_{0,\bar{0}}$ has a fixed point of order $2k$ and the dependence on the parameter $E$ is monotone. Under reasonable assumptions, we show that the rotation number $\rho(E)$ exhibits Lifshitz tail decay with exponent $-\frac{2k - 1}{2k}$, 
\[ \lim_{E \downarrow 0} \frac{\ln(-\ln(\rho(E) - \rho(0)))}{\ln(E)} = -\frac{2k-1}{2k}. \] The exponent is determined by the passage time through a parabolic bottleneck. A full rotation requires on the order of $E^{-\frac{2k - 1}{2k}}$ successive small perturbations, and the probability of such a streak decays exponentially as a function of its length. 
When $k=1$, the exponent is $-1/2$, and we recover as a corollary a purely dynamical proof of Lifshitz tail asymptotics at the spectral edges of the one-dimensional Anderson model.
\end{abstract}
\maketitle

\section{Introduction}
The dynamics of random circle maps and spectral properties of random Schr\"odinger operators share many structural similarities.  
This is not a coincidence: transfer matrices of Schr\"odinger operators act projectively on the circle, and the rotation number of the resulting cocycle is directly related to the integrated density of states (IDS), which measures the proportion of eigenvalues below a given energy. Recent work has exploited this connection to give dynamical proofs of regularity results for the IDS, including log-H\"older continuity and H\"older continuity \cite{GorKle2024, DuaGorKle2025}, both originally established by spectral methods \cite{CraSim1983, LeP1984, BourgainKlein}. These dynamical proofs do more than recover known theorems---they establish these results for the case of random circle diffeomorphisms, going far beyond the spectral setting, and illuminate the underlying mechanisms.

The present paper extends this approach to Lifshitz tails. Near spectral edges, the IDS of a random Schr\"odinger operator decays exponentially, a phenomenon predicted by Lifshitz \cite{Lifshitz1965} and established by various methods \cite{DonVar1975, Simon1985,KirMar1983}; see \cite{Kirsch2008,KirMet2007} for background and further developments. In the one-dimensional Anderson model, if \(E_-\) and \(E_+\) denote the lower and upper spectral edges, this takes the logarithmic form
\[
\lim_{E \downarrow E_-}
\frac{\ln(-\ln k(E))}{\ln(E-E_-)}
=
\lim_{E \uparrow E_+}
\frac{\ln(-\ln(1-k(E)))}{\ln(E_+-E)}
= -\frac12 .
\]
Equivalently, near the lower edge, \(k(E)\) is of order \(\exp(-c(E-E_-)^{-1/2})\), up to constants on the logarithmic scale. The standard heuristic is probabilistic: an eigenvalue close to the spectral edge can only exist if the random potential is unusually small over a large region, and such events are exponentially rare in the region's volume.
The exponent $-1/2$ enters because an eigenvalue at distance $\eps$ from the edge requires the potential to be small on a region of diameter $\sim \eps^{-1/2}$. 

We establish Lifshitz tails in a purely dynamical setting for random circle diffeomorphisms with no reference to Schr\"odinger operators. Consider a one-parameter family of random circle diffeomorphisms $g_{E,y}:\s^1 \rightarrow \s^1$, where $E\in\bbR$ is a deterministic parameter and $y\in\bbR^d$ is sampled i.i.d.\ from a compactly supported measure $\mu$. Let $\rho(E)$ denote the corresponding rotation number. We assume monotonicity in $E$, so $\rho(E)$ is nondecreasing. In many monotone families, $\rho(E)$ can have intervals of constancy (plateaus, corresponding to mode-locking; see \cite{Gourmelon}) alternating with intervals of growth. We focus on the growth right after such a plateau. We assume that at $E=0$, $g_{0,\bar{0}}$ has a parabolic fixed point of order $2k$, that is, the graph of the lift is tangent to the diagonal and curves away with order $2k$ (see Fig.~\ref{pic1} for its lift $G_{0,\bar{0}}$). Under the full assumptions of Section~\ref{sec:MainResults} (including the assumptions on the random parameter $y$ and condition~\ref{fixedPoint}), $E=0$ is the right endpoint of a plateau. Our main result, Theorem~\ref{thm:main} below, describes the rate at which $\rho$ lifts off for small $E>0$. Namely, the increment $\rho(E)-\rho(0)$ has Lifshitz tail asymptotics of order $\exp\left(-c E^{-\frac{2k - 1}{2k}}\right)$ as $E\downarrow 0$.

\begin{figure}[h]
    \centering
    \includegraphics[width=0.5\linewidth]{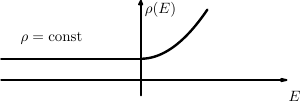}
    \caption{$\rho(E)$}
    \label{fig:plateau}
\end{figure}

The behavior of the rotation number is governed by passage time through the parabolic bottleneck. Orbits slow down dramatically when passing through a neighborhood of the fixed point, and it requires on the order of $E^{-\frac{2k - 1}{2k}}$ iterates to get through said neighborhood. For an orbit to pass through the bottleneck and contribute to the rotation number, the random perturbations must remain small throughout this passage. The probability of such a streak decays exponentially in its length, yielding exponential decay with exponent $-\frac{2k - 1}{2k}$.
As our main result, we prove that under natural assumptions (see Section~\ref{sec:MainResults}) the rotation number $\rho(E)$ satisfies:
\[
\lim_{E \downarrow 0} \frac{\ln(-\ln(\rho(E) - \rho(0)))}{\ln(E)} = -\frac{2k-1}{2k}.
\]

As a corollary, we recover Lifshitz tails for the one-dimensional Anderson model, which describes a quantum particle on $\bbZ$ subject to an i.i.d.\ random potential (see Section~\ref{sec:Anderson} for precise definitions). The transfer matrices associated with the eigenvalue equation act projectively on the circle, and at the spectral edge, this action has a parabolic fixed point of order $2$, placing us exactly in the setting of our main result, Theorem~\ref{thm:main}. This gives a new proof of a classical result, but more importantly, it identifies the source of the exponent: the square-root passage time through a parabolic bottleneck.  

We expect our methods to generalize. Within the Anderson model itself, Lifshitz tails occur not only at spectral edges but also at the boundaries of internal gaps \cite{int01, Simon1987, int03, int04}.
Beyond the Anderson model, natural candidates include models with an ergodic background potential and random noise \cite{AviDamGor2023,DamFilGoh2022,DamGor2022}. The recent construction of an analogue of the rotation number for Schr\"odinger operators on a strip \cite{LiWu} suggests that our methods may apply in that setting as well. 

\medskip
The paper is organized as follows. Section~\ref{sec:MainResults} introduces the setting and main theorem. Section~\ref{sec:Bottleneck} develops bottleneck estimates that control the passage time near the parabolic point. Section~\ref{sec:MainProof} proves the main theorem, and Section~\ref{sec:Anderson} derives Lifshitz tails for the one-dimensional Anderson model.

\section{Setting and Main Results}\label{sec:MainResults}

\subsection{Dynamical Setting}

We consider a one-parameter family of random circle diffeomorphisms $g_{E,y}:\s^1 \to \s^1,$ where $E\in \bbR$ is a deterministic parameter and $y \in \bbR^d$ is a random parameter sampled independently at each iteration; see \ref{compSupp} -- \ref{fixedPoint}. Identify the circle with $\s^1 = \bbR / \bbZ$, with the standard projection $\pi:\bbR\to \s^1$. We assume that the unperturbed map $g_{0, \bar{0}}$ has a unique fixed point, which we place at $\pi(0)$. Furthermore, we pick a lift $G_{E, y}: \bbR \to \bbR$ of the family $g_{E,y}$ such that $\pi \circ G_{E, y} = g_{E, y} \circ \pi$ and $G_{0,\bar{0}}(0) = 0$. The assumptions below isolate the properties of this lifted family that will be used in the proof. As we show in Section \ref{sec:Anderson}, Schr\"odinger cocycles satisfy these assumptions. A simple model example, for $d=1$ and $k=1$, is given by 
\begin{equation*}
    G_{E, y}(x) = x + \frac{\sin^2(\pi x)}{5} + E - y,
\end{equation*}
see Figures~\ref{pic1} and~\ref{pic2} below.

We impose the following conditions on the lift:
\begin{enumerate} [label={\bf(G\arabic*)}]
    \item Regularity: $G_{E, y} (x)$ is $C^{1}$ in $E,$ $y,$ and $x;$ \label{analyt}
    \item Order $2k$ tangency: there exist $c_1, c_2 > 0$ and $\delta_0 > 0$ such that for $x \in [-\delta_0, \delta_0]$ we have
    \begin{equation*}
        x + c_1 x^{2k} \le G_{0, \bar{0}} (x) \le x + c_2 x^{2k}.
    \end{equation*} \label{derder:x}
    \item Monotonicity in $E$:
    \begin{equation*}
        \left. \frac{\partial G_{E, y} (x)}{\partial E} \right|_{E = 0, y = \bar{0}, x = 0} > 0,  
    \end{equation*}
    and
    \begin{equation*}
        \frac{\partial G_{E, y} (x)}{\partial E} \ge 0 \quad \text{for every $y \in \bbR_{+}^d$, $E \in \bbR$, and $x \in \bbR$;}
    \end{equation*}\label{der:E}
    \item Monotonicity in $y$: for every $j \in \{1, 2, \ldots, d\}$
    \begin{equation*}
        \left. \frac{\partial G_{E, y} (x)}{\partial y_j} \right|_{E = 0, y = \bar{0}, x = 0} < 0,
    \end{equation*}
    and
    \begin{equation*}
        \frac{\partial G_{E, y} (x)}{\partial y_j} \le 0 \quad \text{for every $y \in \bbR_{+}^d$, $E \in \bbR$, and $x \in \bbR$.}
    \end{equation*}\label{der:y}
\end{enumerate}
In other words, the graph of $G_{0,0}$ is tangent to the diagonal at the origin and curves away parabolically of order $2k$; increasing $E$ shifts orbits forward, while increasing any component of $y$ shifts them backward.

\begin{figure}[htbp] 
\centering 
\begin{minipage}{0.49\textwidth} 
    \centering
    \includegraphics[width=0.75\linewidth]{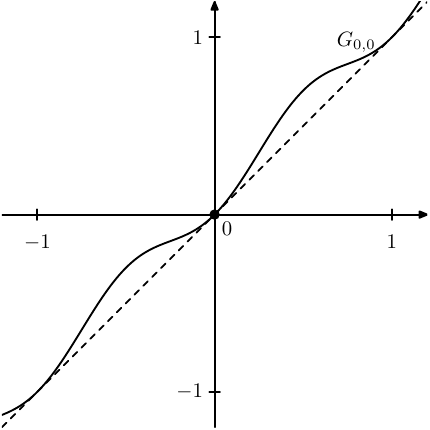} 
    \captionof{figure}{$G_{0,\bar{0}}(x)$} 
    \label{pic1}
\end{minipage} 
\hfill 
\begin{minipage}{0.49\textwidth} 
    \centering 
    \includegraphics[width=0.75\linewidth]{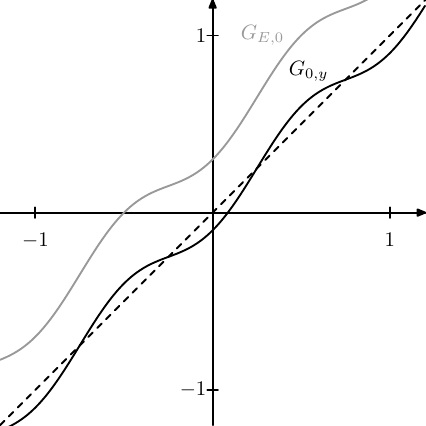}
    {\captionsetup{width=\textwidth}
    \captionof{figure}{$G_{E>0,\bar{0}}(x)\text{ and }G_{0,y>0}(x)$}
    \label{pic2}}
\end{minipage}
\end{figure}
Let $\mu$ be a probability measure on $\bbR^d$, and let $\bomega=(\omega_1, \omega_2, \ldots)$ be a sequence of i.i.d.\ random variables with distribution $\mu$. We write 
\begin{align*}
    g^N_{E,\bar{\omega}}
    &:= g_{E,\omega_N} \circ \cdots \circ g_{E,\omega_1}; \\
    G^N_{E,\bar{\omega}}
    &:= G_{E,\omega_N} \circ \cdots \circ G_{E,\omega_1},
\end{align*}
and let $\sigma$ denote the left shift on 
$(\bbR^d)^{\bbN}$, given by $(\sigma(\bomega))_n = \omega_{n + 1}.$ We write $\bbP=\mu^\bbN$ for the probability on the sequence space. We impose the following conditions on $\mu$:
\begin{enumerate}[label={\bf(M\arabic*)}]
    \item $\supp(\mu)$ is a compact subset of $[0,\infty)^d$ containing at least two points;\label{compSupp}
    \item there exist $C, l > 0$ such that for every $\eps > 0$,
    \begin{equation*}
        \mu \left( [0, \eps]^d \right) > C \eps^l;
    \end{equation*} \label{ineq:dens}    
    \item there exists a point $x_{*} \in \bbR$ such that for every $y \in \supp(\mu)$ we have \[G_{0, y} (x_{*}) \, >\, x_{*}.\] \label{fixedPoint}
\end{enumerate}

For the discussion of the meaning and necessity of assumptions \ref{compSupp}--\ref{fixedPoint}, see Remark~\ref{rem:assumptions}.

The rotation number measures the average rate at which orbits wind around the circle. The following is standard; see, e.g., \cite[Section~5]{Herman1983}, \cite{LiLu2008}, \cite{Ruelle1985}, or \cite[Appendix~A]{GorKle2021}.
\begin{prop} \label{prop:rotation}
    For every $E$ there exists a constant $\rho(E)$, called the rotation number, such that for $\bbP$-a.e. $\bomega$ and every $x \in \bbR$ one has
    \begin{equation*}
        \rho (E)= \lim_{n \to \infty} \frac{G^n_{E, \bomega} (x) - x}{n}.
    \end{equation*}
\end{prop}

Our assumptions, in particular \ref{fixedPoint}, guarantee that $E = 0$ is a right endpoint of a plateau for $\rho(E)$. Namely, for sufficiently small $E > 0,$ we have $\rho(-E) = \rho(0).$ 

\subsection{Main Result for the Dynamics on the Circle}

Our main result establishes Lifshitz tails for the rotation number in a neighborhood of $E = 0$:
\begin{theorem} \label{thm:main}
Let $g_{E, y}$ and $\mu$ satisfy assumptions \ref{analyt} -- \ref{der:y} and \ref{compSupp} -- \ref{fixedPoint}. Then
      \begin{equation*}
        \lim_{E \downarrow 0} \frac{\ln(-\ln(\rho(E)-\rho(0)))}{\ln(E)} = -\frac{2k-1}{2k}.
    \end{equation*}
\end{theorem}

\begin{remark}
    Theorem \ref{thm:main} suggests several directions for further study: 
    \begin{itemize}

        \item If we replace i.i.d.\ diffeomorphisms with an arbitrary skew product over an ergodic base and a circle as the fiber, the rotation number is still well defined. Under what conditions would we observe Lifshitz tails in this more general setting?

        \item Generalizations of the rotation number arise in several settings. A continuous-time analogue of it is used in the study of linear Hamiltonian systems (see \cite{Johnson} and references therein). Another analogue of it has recently been found for symplectic cocycles (see \cite{LiWu}) and used to answer some questions in spectral theory of discrete Schr\"odinger operators on a strip. A natural question is whether Lifshitz tails occur in these settings as well.
    \end{itemize}
\end{remark}

\begin{remark}\label{rem:assumptions}
    The intuition behind assumptions \ref{compSupp} -- \ref{fixedPoint} is the following:
    \begin{itemize}
        
        \item In \ref{compSupp}, the most important part is that the support of $\mu$ has to be nontrivial. Otherwise, the dynamics stops being random, and the rotation number for a deterministic family of diffeomorphisms is known to behave differently. For example, for the Schr\"odinger cocycle with zero potential, the rotation number can be computed explicitly, and it does not exhibit Lifshitz tails. The remaining parts of \ref{compSupp} are mostly technical. In particular, one can replace the compactness assumption for $\supp(\mu)$ with uniform estimates for the derivatives of $G_{E, \bar{y}}$.

        \item Assumption \ref{ineq:dens} is used to establish the lower bound. If one were to make the density of $\mu$ at zero subpolynomial, for example $\mu([0, \eps]) \approx e^{-1/\eps}$, the last computation in Section~\ref{sec:lower} would only guarantee that
        \begin{equation*}
            \lim_{E\downarrow 0} \frac{\ln(-\ln(\rho(E)-\rho(0)))}{\ln(E)} \geq -\frac{4k-1}{2k},
        \end{equation*}
        which does not determine the exact decay rate.

        \item Finally, \ref{fixedPoint} guarantees that the point $E = 0$ is indeed a right end of a plateau for the rotation number $\rho(E)$. Indeed, if there were a parameter value $\bar{y} \in \supp(\mu)$ such that for any $x$ one had $G_{0, \bar{y}} (x) < x,$ then by compactness in a finite number of steps $n$ with some positive probability the trajectory would make a full rotation in the negative direction ($G^n_{0, \bomega} (x) < x - 1$ for some positive measure set of $\bomega$). In that case, for any small positive value $E > 0$, an argument similar to the one in Section~\ref{sec:lower} would show that $\rho(-E) < \rho(0).$ Thus $E = 0$ cannot be a right endpoint of a plateau.
    \end{itemize}
\end{remark}

\subsection{Application to the Anderson Model}

As an application of Theorem~\ref{thm:main}, we derive Lifshitz tails for the one-dimensional Anderson model. Consider the random Schr\"odinger operator on $\ell^2(\bbZ)$ given by 
\[
H_\omega \psi (n) = \psi(n+1) + \psi(n-1) + V_\omega(n)\psi(n),
\]
where $\{V_\omega (n)\}_{n \in \bbZ}$ is an i.i.d.\ family of random variables with common distribution $\mu$, a  probability measure whose support is $[a,b]$ with $a\neq b.$ The almost-sure spectrum is $\Sigma=[a-2,b+2]$. We denote the integrated density of states by $k(E)$; informally, $k(E)$ measures the proportion of eigenvalues below energy $E$. The associated transfer matrices 
\[
A_{E,\omega} = \begin{bmatrix} E - V_\omega(n) & -1 \\ 1 & 0 \end{bmatrix} \in \mathrm{SL}(2,\bbR)
\]
act projectively on $\mathbb{RP}^1\simeq \s^1$ as random circle diffeomorphisms. At the spectral edges, this projective action has a quadratic parabolic fixed point. We use the projective normalization of the rotation number, for which the IDS satisfies $k(E)=1 - \rho(E).$ This is the usual relation between the integrated density of states and the rotation number after passing to $\mathbb{RP}^1$, where opposite directions are identified; compare the double-cover normalization in \cite{DelyonSouillard1983}. Applying Theorem~\ref{thm:main} yields the
following. 

\begin{theorem}\label{thm:AndersonLifshitz}
If $a\neq b$ and there exist $C,l$ such that $\mu([a,a+\eps))\geq C\eps^l$ and $\mu((b-\eps,b]) \geq C\eps^l$ for sufficiently small $\eps$, then the integrated density of states $k$ satisfies
    \begin{equation*}
        \lim_{E \downarrow E_-} \frac{\ln\bigl(-\ln k(E)\bigr)}{\ln(E - E_-)} =
        \lim_{E \uparrow E_+} \frac{\ln\bigl(-\ln(1 - k(E))\bigr)}{\ln(E_+ - E)} = -\frac{1}{2},
    \end{equation*}
    where $E_-=a-2$ and $E_+=b+2$ are the spectral edges. 
\end{theorem}

\section{Non-Stationary Bottleneck Estimates}\label{sec:Bottleneck}
This section establishes upper and lower bounds on the number of iterates required to pass through a neighborhood of a fixed point of order $2k$. These estimates are the key input for Section~\ref{sec:MainProof}, where they determine how many consecutive small perturbations are needed for a full rotation.
To motivate the exponent, consider the following heuristic.
Near the fixed point, the map behaves like $x \mapsto x + \lambda x^{2k} + \eps$. The critical region is the \emph{collar} $|x| \le \eps^{\frac{1}{2k}}$. Inside the collar, the orbit advances on the order of $\eps$ per step and must traverse a distance on the order of $\eps^{\frac{1}{2k}}$, so the crossing time is on the order of $\eps^{\frac{1}{2k}}/\eps = \eps^{-\frac{2k - 1}{2k}}$. Sufficiently near the fixed point, but outside the collar, the term $\lambda x^{2k}$ dominates,  contributing transit time of the same order. The following two lemmas make this precise.
For a sequence of maps $f_j : \bbR \to \bbR$ we write $f^n = f_n \circ f_{n-1} \circ \cdots \circ f_1$.
 
\begin{figure}[htbp]
\centering
\includegraphics[width=0.5\textwidth]{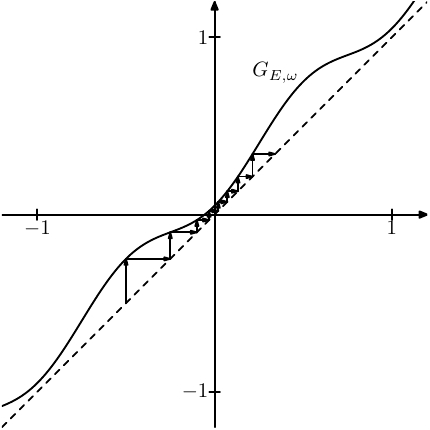}
\caption{A trajectory passing through the bottleneck near the fixed point. The orbit decelerates as it approaches the collar $|x| \le \eps^{\frac{1}{2k}}$ (Phase~1), drifts slowly through (Phase~2), and accelerates away (Phase~3).}
\label{fig:trajectory}
\end{figure}
 
\begin{lemma}\label{lemma:bottleneckLB}
    Let $f_j$ be a sequence of monotone increasing functions on $[-\delta, \delta]$ for some $\delta > 0$. If for some $\lambda > 0$, any $j \in \bbN$, and any $x \in [-\delta, \delta]$ we have
    \begin{equation*}
        f_j(x) > x + \lambda x^{2k} + \eps,
    \end{equation*}
    then there exist a constant $C_1$ (that does not depend on $\eps$) and $\eps_0 > 0$ such that for all $0 < \eps < \eps_0$ and $N_1 = \lceil C_1 \eps^{-\frac{2k - 1}{2k}} \rceil$ we have
    \begin{equation*} \label{stepsLB}
        f^{N_1}(-\delta) \ge \delta.
    \end{equation*}
\end{lemma}
 
\begin{proof}
    We split the orbit into three phases. Let $x_0 = -\delta$ and $x_{m+1} = f_{m+1}(x_m)$.
 
    \textit{Phase 1: from $-\delta$ to $-\eps^{\frac{1}{2k}}$.}
    Let $M_1$ be the smallest integer such that $f^{M_1}(-\delta) \ge -\eps^{\frac{1}{2k}}$. Since $x \le -\eps^{\frac{1}{2k}}$, we have $f_j(x) - x > \lambda x^{2k} + \eps > \lambda x^{2k}$. In particular, $x \mapsto 1/(\lambda x^{2k})$ is increasing on $(-\delta, -\eps^{\frac{1}{2k}})$, so the left Riemann sum is bounded by the integral:
    \begin{align*}\label{eq:phase1}
        & \sum_{m=0}^{M_1 - 2} \frac{x_{m+1} - x_m}{\lambda x_m^{2k}}\\ & < \int_{-\delta}^{-\eps^{\frac{1}{2k}}} \frac{d x}{\lambda x^{2k}}\\ & = \frac{1}{\lambda(2k-1)} \left( \eps^{-\frac{2k - 1}{2k}} - \delta^{-(2k-1)} \right).
    \end{align*}
    Since each summand exceeds $1$ (because $x_{m+1} - x_m > \lambda x_m^{2k}$), we have \[M_1 \le \eps^{-\frac{2k - 1}{2k}} / (\lambda(2k-1))+1.\]
 
    \textit{Phase 2: from $-\eps^{\frac{1}{2k}}$ to $\eps^{\frac{1}{2k}}$.}
    Let $M_2 = \lceil 2\eps^{-\frac{2k - 1}{2k}} \rceil$. Since $f_j(x) - x > \eps$ for all $x$, we have that for any $x$ in this interval, $f^{M_2}(x) > \eps^{\frac{1}{2k}}$.
 
    \textit{Phase 3: from $\eps^{\frac{1}{2k}}$ to $\delta$.}
    Again we rely upon $f_j(x) - x > \lambda x^{2k}$. Let $M_3$ be the smallest integer such that $f^{M_3}(x) > \delta$ for all $x$ in this interval. In this regime, $x^{2k}$ grows with $x$. We estimate the transit time by a dyadic decomposition. For any $x \ge L$, the increment exceeds $\lambda L^{2k}$, so the orbit doubles from $L$ to $2L$ in at most $\lceil 1/(\lambda L^{2k-1}) \rceil$ steps. Letting $n$ be the number of times $\eps^{\frac{1}{2k}}$ must be doubled to exceed $\delta$, we have
    \begin{align*}\label{eq:phase3}
        M_3 & \le \sum_{j=0}^{n} \left\lceil \frac{1}{\lambda (2^j \eps^{\frac{1}{2k}})^{2k-1}} \right\rceil\\ & < n + \frac{\eps^{-\frac{2k - 1}{2k}}}{\lambda} \sum_{j=0}^{\infty} 2^{-j(2k-1)}\\ & = n + \eps^{-\frac{2k - 1}{2k}} \frac{2^{2k-1}}{\lambda(2^{2k-1} - 1)}.
    \end{align*}
    We observe that $n = \lceil \log_2(\delta / \eps^{\frac{1}{2k}}) \rceil$, which is dominated by $\eps^{-\frac{2k - 1}{2k}}$ as $\eps \to 0$.
 
    Combining all three phases, we can choose
    \begin{equation*}
        C_1 = \frac{1}{\lambda(2k-1)} + 2 + \frac{2^{2k-1}}{\lambda(2^{2k-1} - 1)} + 1,
    \end{equation*}
    which gives the desired property.
\end{proof}

\begin{lemma}\label{lemma:bottleneckUB}
    Let $f_j$ be a sequence of monotone increasing functions on $[-\delta, \delta]$ for some $\delta > 0$. If for some $\Lambda > 0$, any $j \in \bbN$, and any $x \in [-\delta, \delta]$ we have
    \begin{equation*}
        f_j(x) < x + \Lambda x^{2k} + \eps,
    \end{equation*}
    then there exist a constant $C_2$ (not dependent on $\eps$) and $\eps_0 > 0$ such that for all $0 < \eps < \eps_0$ and $N_2 = \lfloor C_2 \eps^{-\frac{2k - 1}{2k}} \rfloor$ we have
    \begin{equation*} \label{stepsUB}
        f^{N_2}(-\eps^{\frac{1}{2k}}  ) < 0.
    \end{equation*}
\end{lemma}
 
\begin{proof}
    Choose $\eps_0$ small enough so that $\eps^{\frac{1}{2k}} < \delta$. For $x \in [-\eps^{\frac{1}{2k}}, 0]$ we have $|x|^{2k} \le \eps$, hence
    \begin{equation*}
        f_j(x) - x < \Lambda x^{2k} + \eps \le (\Lambda + 1)\eps.
    \end{equation*}
    Starting from $-\eps^{\frac{1}{2k}}$, each step advances the orbit by at most $(\Lambda + 1)\eps$, so reaching $0$ requires at least
    \begin{equation*}
        \frac{\eps^{\frac{1}{2k}}}{(\Lambda + 1)\eps} = \frac{1}{\Lambda + 1} \, \eps^{-\frac{2k - 1}{2k}}
    \end{equation*}
    steps. Setting $C_2 = (\Lambda + 1)^{-1}$ establishes $f^{N_2}(-\eps^{\frac{1}{2k}}) < 0$.
\end{proof}

\section{Proof of the Main Result}\label{sec:MainProof}
The proof consists of a lower bound and an upper bound. For the lower bound, we estimate the frequency of full rotations by identifying blocks of consecutive small perturbations that allow the orbit to pass through the bottleneck. For the upper bound, we use the complementary approach and control the contribution of sufficiently large perturbations that keep us from passing through the bottleneck.

Under assumptions \ref{der:E}, \ref{der:y}, and \ref{fixedPoint}, we have
\begin{equation}\label{eq:no-backtracking}
    G^n_{E,\bomega}(x)>-1
    \quad \text{for every } x\ge 0,\ 
    \bomega\in\supp(\mu)^{\bbN},\ E\ge 0,\ n\in\bbN,
\end{equation}
and $\rho(0)=0$. Indeed, we can choose the representative
$x_*\in(-1,0)$ satisfying \ref{fixedPoint}. For $E=0$, the
interval $[x_*,0]$ is invariant: the lower bound follows from monotonicity in
$x$ and $G_{0,y}(x_*)>x_*$, while the upper bound follows from monotonicity in
$y$ and $x$, together with $G_{0,\bar 0}(0)=0$. Hence the lift orbit is bounded
at $E=0$, so $\rho(0)=0$.

Moreover, by \ref{fixedPoint}, monotonicity in $E$, and periodicity of the lift,
\[
G_{E,y}(x_*+m)>x_*+m
\quad
\text{for all } E\ge0,\ y\in\supp(\mu),\ m\in\bbZ.
\]
In particular, for $x\ge0$, iteration gives
$G^n_{E,\bomega}(x)>x_*>-1$.

\begin{lemma} \label{lm:GBounds}
    There are $\delta > 0$ and positive constants $c_1, c_2, \gamma_E,  \Gamma_E, \gamma_y, \Gamma_y$  such that for every $x \in [- \delta, \delta]$, $E \in [0, \delta]$ and $y \in [0, \delta]^d$,
    \begin{equation*}
        x + c_1 x^{2k} + \gamma_E E - d \Gamma_y \|y\| \,\le\, G_{E, y} (x) \,\le\, x + c_2 x^{2k} + \Gamma_E E - \gamma_y \|y\|,
    \end{equation*}
    where $\|y\| = \max_{j = 1}^{d} |y_j|$.
\end{lemma}

\begin{proof}
By \ref{derder:x} we have for small enough $\delta,$
\[
    x+c_1 x^{2k}\le G_{0,\bar{0}}(x)\le x+c_2 x^{2k}
    \quad\text{for } x\in[-\delta,\delta].
\]
By continuity of the derivatives and the sign assumptions in \ref{der:E} and \ref{der:y}, we can choose $\delta$ small enough so that
\[
    \gamma_E\le \frac{\partial G_{E,y}(x)}{\partial E} \le \Gamma_E,
    \qquad
    -\Gamma_y\le \frac{\partial G_{E,y}(x)}{\partial y_j} \le -\gamma_y
\]
on this neighborhood, for suitable positive constants. The mean value theorem in $E,y_1,\ldots,y_d$ gives the desired result.
\end{proof}

\subsection{Lower Bound}\label{sec:lower}
Choose $b>0$ small enough that $d\,\Gamma_y b \le \gamma_E/2$. For $E>0$, call a block $(y_1,\ldots,y_N)$ \emph{good} if $y_j \in[0,bE]^d$ for all $j.$

\begin{lemma}\label{l:lower}
For all sufficiently small $E>0$, if $N$ is a sufficiently large multiple of $E^{-\frac{2k-1}{2k}}$, then every good block of length $N$ satisfies
\[
G_{E,y_1\dots y_N}^N (x_*) \ge x_*+1.
\]
Otherwise, $G_{E,y_1\dots y_N}^N (x_*) > x_*.$
\end{lemma}

\begin{proof}
For the first claim, choose $\delta$ from Lemma~\ref{lm:GBounds} with $x_*<-\delta<0<\delta<x_*+1$. Since $g_{0,\bar{0}}$ has no fixed point on the circle except $0$, the function $G_{0,\bar{0}}(x)-x$ is bounded away from $0$ on $ [x_*,x_*+1]\setminus(-\delta,\delta).$ By continuity, for all sufficiently small $E$ and all good $y$, this difference is still bounded away from $0$ on the same set. Thus, only a bounded number of good iterates is needed to move from $x_*$ to $-\delta$, and later from $\delta$ to $x_*+1$. Inside $[-\delta,\delta]$, for a good letter and small $E$,
\[
    G_{E,y}(x)
    \ge x+c_1 x^{2k}+\gamma_E E-d\,\Gamma_ybE
    \ge x+c_1 x^{2k}+\frac{\gamma_E}{2}E.
\]
We then apply Lemma~\ref{lemma:bottleneckLB}, using with $\eps=(\gamma_E/2)E$, to see that a constant multiple of $E^{-\frac{2k-1}{2k}}$ such iterates sends $-\delta$ to at least $\delta$. Increasing the multiplicative constant in $N$ absorbs the bounded number of iterates outside the bottleneck.

The second claim follows from \eqref{eq:no-backtracking} and monotonicity in $x$.
\end{proof}

The lemma immediately implies:
\begin{coro}\label{c:lower}
For all sufficiently small $E>0$, if $N=\lceil A E^{-\frac{2k-1}{2k}}\rceil$ with $A$ large enough, then
    \[
        \rho(E)\ge \frac{1}{N} \Prob(\text{block $(y_1,\dots, y_N)$ is good}) \ge \frac{1}{N} (C(bE)^l)^{N}.
    \]
\end{coro}

\begin{proof}
    It follows from Lemma~\ref{l:lower} that for a composition of $m$ blocks, of total length $n=mN$, one has 
    \[
        G^{n}_{E,y_1\dots y_n}(x_*)\ge x_* + \# \{j\le m  \mid \text{block $y_{jN+1},\dots ,y_{(j+1)N}$ is good}\}.
    \]
    Dividing by $n$, passing to the limit, and applying the law of large numbers, one obtains the statement of the corollary.
\end{proof}

\subsection{Upper Bound}
We will now construct blocks used to prove the upper bound. Since $0\in\supp(\mu)$ by \ref{ineq:dens}, and $\supp(\mu)$ contains at least two points, choose $\hat y\in\supp(\mu)$ with $\hat y\ne\bar 0$. For some
coordinate $j$ we have $\hat y_j>0$. Since
\[
    \left.\frac{\partial G_{E,y}(x)}{\partial y_j}\right|_{E=0,y=\bar 0,x=0}<0
\]
and $G_{0,y}(0)$ is nonincreasing in each coordinate of $y$, we have
$G_{0,\hat y}(0)<0$. Hence there exists a small $\eta>0$ such that
\[
    G_{0,\hat y}(\eta)<-2\eta .
\]
By continuity, there is a neighborhood $U$ of $\hat y$ such that, for all
sufficiently small $E>0$,
\[
    G_{E,y}(\eta)<-\eta
    \quad\text{for every } y\in U.
\]
Set $p_1:=\mu(U)>0$.
We call letters in $U$ \emph{bad}. A block is \emph{bad} if it contains at
least one bad letter.

\begin{lemma}\label{l:short-time}
For all sufficiently small $E>0$, if $n$ is at most a sufficiently small
multiple of $E^{-\frac{2k-1}{2k}}$, then for every sequence
$y_1,\ldots,y_n\in\supp(\mu)$ and every $x\le0$,
\[
    G^n_{E,y_1\dots y_n}(x)<\eta.
\]
If, moreover, $x\le-\eta$, then
\[
    G^n_{E,y_1\dots y_n}(x)<0.
\]
\end{lemma}

\begin{proof}
Choose $\eta<\delta$, with $\delta$ as in Lemma~\ref{lm:GBounds}.
Since the maps are increasing in $x$ and nonincreasing in $y$,
\[
    G^n_{E,y_1\dots y_n}(x)\le G^n_{E,\bar 0}(0)
    \quad\text{for } x\le0.
\]
By Lemma~\ref{lm:GBounds},
\[
    G_{E,\bar 0}(z)\le z+c_2z^{2k}+\Gamma_EE
    \quad\text{for } z\in[0,\delta].
\]
As long as $0\le z\le E^{1/(2k)}$, the increment is bounded by a constant
multiple of $E$. Therefore, if the multiplicative constant in the upper bound
on $n$ is chosen sufficiently small, the orbit of $0$ cannot reach $E^{1/(2k)}$ in
$n$ steps. Since $E^{1/(2k)}<\eta$ for small $E$, this proves the first claim.

For the second claim, Lemma~\ref{lm:GBounds} gives
\[
    G_{E,y}(z)\le z+c_2z^{2k}+\Gamma_EE
    \quad\text{for } z\in[-\delta,\delta],\ y\in\supp(\mu).
\]
Lemma~\ref{lemma:bottleneckUB}, applied with
$\eps=\Gamma_EE$, shows that a trajectory starting at
$-(\Gamma_EE)^{1/(2k)}$ cannot reach $0$ in fewer than a constant multiple of
$E^{-\frac{2k-1}{2k}}$ steps. Since
$-\eta<-(\Gamma_EE)^{1/(2k)}$ for small $E$, monotonicity in $x$ gives the
same conclusion for every $x\le-\eta$.
\end{proof}

\begin{figure}
        \centering
        \includegraphics[width=0.5\linewidth]{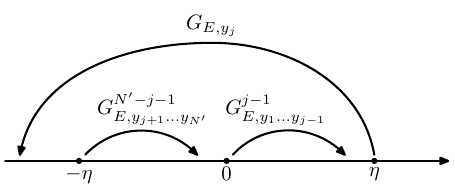}
        \caption{The orbit of a point $x$ near $0$ under the composition $G$ associated with a bad block of length $N'$, with bad letter $y_j$.}
        \label{fig:bottleneck-blocked}
    \end{figure}

\begin{lemma}\label{l:upper}
For all sufficiently small $E>0$, let
\[
    N'=\left\lfloor aE^{-\frac{2k-1}{2k}}\right\rfloor,
\]
where $a>0$ is sufficiently small. Then for every block of length $N'$, the associated
composition maps every $x\le0$ to a point less than $1$. If the block is bad, then the associated composition maps every $x\le0$ to a point less than $0$.
\end{lemma}

\begin{proof}
The first statement follows from Lemma~\ref{l:short-time}, since $\eta<1$.

Suppose the block is bad, and let $y_j$ be a bad letter in the block. By
Lemma~\ref{l:short-time}, the prefix preceding $y_j$ sends any $x\le0$ to some
$x'<\eta$. Since $G_{E,y_j}$ is increasing and $y_j$ is bad,
\[
    G_{E,y_j}(x')\le G_{E,y_j}(\eta)<-\eta.
\]
The suffix has length at most $N'$, so the second part of
Lemma~\ref{l:short-time} keeps the image below $0$.
\end{proof}

The lemma immediately implies:
\begin{coro}\label{c:upper}
For all sufficiently small $E>0$, if
$N'=\lfloor aE^{-\frac{2k-1}{2k}}\rfloor,$ with $a>0$ sufficiently small, then
\[
    \rho(E)
    \le \frac{1}{N'}\Prob(\text{block $(y_1,\dots ,y_{N'})$ is not bad})
    \le \frac{1}{N'}(1-p_1)^{N'}.
\]
\end{coro}

\begin{proof}
It follows from Lemma~\ref{l:upper} that for a composition of $m$ blocks, of total length $n=mN'$, one has
\[
    G^n_{E,y_1\dots y_n}(0)
    \le
    \#\{1\le j\le m:\text{ block }y_{(j-1)N'+1}\dots y_{jN'}\text{ is not bad}\}.
\]
Indeed, after each block, we subtract the current integer displacement and use
periodicity, so that the composition associated with the next block is applied
to a point $x\le0$. By Lemma~\ref{l:upper}, a bad block contributes no new
full turn, while an arbitrary block contributes at most one. Dividing by $n$,
passing to the limit, and applying the law of large numbers gives the
statement. Since a block is not bad exactly when none of its $N'$ letters lies
in $U$, this probability is $(1-p_1)^{N'}$.
\end{proof}

\subsection{Asymptotics of the Rotation Number}
We are now ready to conclude the proof of our main result, Theorem~\ref{thm:main}.

\begin{proof}[Proof of Theorem~\ref{thm:main}]
Since $\rho(0)=0,$ the lower estimate gives, for $N = \lceil AE^{-\frac{2k-1}{2k}} \rceil$ with $A$ sufficiently large,
 \[
 \rho(E) \ge \frac{1}{N} \left(C (bE)^l \right)^N 
 \]
 and thus 
 \begin{align*}\label{eq:rho-lower}
    -\ln (\rho(E)) & \le  \ln (N) - N ( l \ln (E) + \ln (b^l C))\\ & \le 2 l N | \ln (E)|,     
 \end{align*}
 as both $\ln (N) = o(N \ln (E))$ and $N= o(N \ln (E))$ as $E\to 0$.

 By the choice of $N$, one has 
 \[
 \ln (N)  \sim -\frac{2k-1}{2k} \ln (E) \quad \text{and} \quad \ln (-\ln (E))=o(\ln( N))
 \]
 as $E\to 0$. Changing signs and taking one more logarithm gives
 \begin{align*}
    \ln (-\ln \rho(E)) & \le \ln (-2l \cdot N \ln E)\\ & = \ln 2l + \ln (-\ln E) + \ln N\\ & = \ln N \cdot (1+o(1)).
 \end{align*}
 Thus,
 \begin{equation}\label{eq:liminf-lower}
    \liminf_{E\downarrow 0}\frac{\ln (-\ln \rho(E))}{\ln E} \ge \lim_{E\downarrow 0} \frac{\ln N}{\ln E} = -\frac{2k-1}{2k}     
 \end{equation}

Similarly, Corollary~\ref{c:upper} gives, for all sufficiently small $E>0$
\begin{align*}\label{eq:rho-upper}
- \ln (\rho(E)) & \ge \ln (N') - N'\ln (1-p_1)\\ & \ge c'N',    
\end{align*}
where $c':= -\frac{1}{2} \ln (1-p_1)>0$.
Taking logarithms and dividing by~$\ln (E)<0$, we obtain
\[
\frac{\ln (-\ln( \rho(E)))}{\ln (E)} \le \frac{\ln (c')+ \ln (N')}{\ln (E)}.
\]
As $\ln (N')\sim -\frac{2k-1}{2k} \ln (E)$, we have
\begin{equation}\label{eq:limsup-upper}
\limsup_{E\downarrow 0} \frac{\ln (-\ln (\rho(E)))}{\ln (E)} \le \lim_{E\downarrow 0} \frac{\ln (N')}{\ln (E)} = -\frac{2k-1}{2k}.
\end{equation}
Combining~\eqref{eq:limsup-upper} with~\eqref{eq:liminf-lower}, we obtain
\[
\lim_{E\downarrow 0} \frac{\ln (-\ln( \rho(E)))}{\ln (E)} = -\frac{2k-1}{2k}.
\]
This completes the proof.
\end{proof}

\section{Proof of the Application to the Anderson Model}\label{sec:Anderson}
The Anderson model, introduced in \cite{Anderson1958}, describes a quantum particle moving on a lattice in the presence of random impurities. It remains a central object of study in the theory of random Schr\"odinger operators.

Lifshitz tails for the continuous model were first established rigorously by Donsker and Varadhan \cite{DonVar1975}, using large-deviation techniques. For the discrete Anderson model, Simon \cite{Simon1985} provided an alternative proof using Dirichlet--Neumann bracketing. Both approaches are fundamentally spectral in nature. Here, we recover Lifshitz tails for the one-dimensional discrete Anderson model as a consequence of Theorem~\ref{thm:main}, thereby giving a purely dynamical proof.
The point is not merely to reprove a known result. The transfer matrices associated to the eigenvalue equation act projectively as circle diffeomorphisms, and at the spectral edges, this projective action has a quadratic parabolic fixed point. The Anderson model thus falls naturally within the scope of our main theorem. Thus, the Lifshitz exponent $-1/2$ for the Anderson model emerges from the same bottleneck geometry that governs general random circle diffeomorphisms near a parabolic point. To make this precise, we recall the setup and verify the hypotheses of Theorem~\ref{thm:main} in the quadratic parabolic case arising at the spectral edge. 

Let $\mu$ be a probability measure on $\bbR$ whose support is $[a,b],$ and set 
\[
\Omega:=\supp(\mu)^{\bbZ} \quad \text{ and } \quad \bbP:=\mu^{\bbZ}.
\]
For $\omega\in \Omega$ let $\{V_{\omega}(n)\}_{n\in\bbZ}$ be a sequence of i.i.d.\ random variables with distribution $\mu$ on the probability space $(\Omega,\bbP)$. The Anderson model is the family of random Schr\"odinger operators $\{H_\omega\}_{\omega\in \Omega}$ acting on $\ell^2(\bbZ)$ by
\begin{equation*}\label{eq:eso}
    [H_\omega \psi] (n) = \psi(n+1)+\psi(n-1)+ V_\omega(n)\psi(n) , \quad \psi \in \ell^2(\bbZ).
\end{equation*} 

Associated with the family $\{H_\omega\}_{\omega \in \Omega}$ is the 
density of states measure $\kappa$, which can be characterized via eigenvalue distributions of finite restrictions of $\{H_\omega\}$:
for any bounded measurable function $g:\bbR \to \bbR$,
\begin{equation*}
\int g \, d\kappa = \lim_{N \to \infty} \frac{1}{N} \mathrm{Tr}\bigl(g(H_\omega) \chi_{[1,N]}\bigr)
\end{equation*}
for $\bbP$-almost every $\omega \in \Omega$. The integrated density of states (IDS) is then
\[ k(E)=\int \chi_{(-\infty,E]}\, d\kappa,\] which measures the proportion of eigenvalues below energy $E$.
For further background on the IDS, see \cite{ESO1}; for a comprehensive treatment of random Schr\"odinger operators, see \cite{ESO2}.
 
A key feature of the Anderson model is the explicit formula for the almost-sure spectrum \cite{KunSou1980},
\begin{equation}\label{eq:andersonSpectrum}
    \Sigma=\sigma(H_\omega)= \supp \mu + [-2,2] \quad \text{ for } \bbP \text{-a.e. } \omega. 
\end{equation}
The spectral edges are $E_-=a-2$ and $E_+=b+2$.
We impose the following assumptions on $\mu$:
\begin{enumerate}[label={\bf(A\arabic*)}]
    \item \label{A1} The measure $\mu$ is not supported at a single point, that is, $a \neq b$;
    \item \label{A2} There exist constants $C_1,l_1 >0$ such that for small $\eps>0$, 
    \begin{equation*}
        \mu([a,a+\eps]) \geq C_1\eps^{l_1};
    \end{equation*}
    \item \label{A3} There exist constants $C_2,l_2 >0$ such that for small $\eps>0$,  
    \begin{equation*}
        \mu([b-\eps,b]) \geq C_2\eps^{l_2}.
    \end{equation*}
\end{enumerate}
These assumptions are satisfied by a large class of distributions. 

For example, if $\mu$ is the uniform distribution on $[a,b]$, then for every $0<\eps<b-a$ one has 
\[
\mu([a,a+\eps]) = \mu([b-\eps,b])=\frac{\eps}{b-a}
\]
so \ref{A2} and \ref{A3} hold with $C_1 =C_2=(b-a)^{-1}$ and $l_1=l_2=1.$

Let $E=E_- + \eps.$ For $v\in \supp(\mu),$ write $v=u+a$, where $u\in [0,b-a]$. Let $T_a$ be the translation sending $v$ to $v-a$ and define
\[
\nu := (T_a)_*\mu.
\]
In the remainder of this section, $E$ denotes the spectral energy in the Anderson model, while $\eps=E-E_-$ plays the role of the deterministic parameter in Theorem~\ref{thm:main}.

For $(\eps,u),$ define the one-step transfer matrix 
\begin{equation}
    A_{\eps,u}=
    \begin{bmatrix}
        (E_- +\eps) - (a+u) & -1 \\
        1 & 0
    \end{bmatrix}
    =
    \begin{bmatrix}
        -2 + (\eps-u) & -1 \\
        1 & 0
    \end{bmatrix}
    \in SL(2,\bbR).
\end{equation}
Fix an orientation-preserving identification $\varphi:\mathbb{RP}^1 \to \s^1=\bbR/\bbZ$ that sends the parabolic fixed point of the projective action of $A_{0,0}$ to $0\in\bbR/\bbZ$.
Let $g_{\eps,u}: \s^1 \rightarrow \s^1$ be the circle diffeomorphism induced by the projective action of $A_{\eps,u}$ through $\varphi$, and let $G_{\eps,u}:\bbR \rightarrow \bbR$ be the lift chosen continuously in $(\eps,u)$ and normalized by $G_{0,0}(0) = 0$. 
We will need the following lemma.
\begin{lemma}\label{lemma:anderson-conditions}
        There exists $r\in(0,1/2)$ such that the family of lifts 
        \[ G_{\eps,u}:\bbR\to\bbR,\qquad |\eps|\le r,\quad 0\le u\le b-a,
        \]
satisfies the hypotheses \ref{analyt}--\ref{der:y} and
\ref{compSupp}--\ref{fixedPoint} of Theorem~\ref{thm:main}, in the quadratic parabolic case, with $\mu$ replaced by its pushforward $\nu$.
\end{lemma}

\begin{proof}
Let
$F_{\eps,u}$ denote the projective action of $A_{\eps,u}$.
We write points of $\mathbb{RP}^1$ as $[\xi:\eta]$ and identify $\mathbb{R} \cup \{ \infty \}$ via $t=\xi/\eta + 1$. Since
\[
A_{\eps,u}
\begin{bmatrix}
\xi\\ \eta
\end{bmatrix}
=
\begin{bmatrix}
(-2+\eps-u)\xi-\eta\\
\xi
\end{bmatrix},
\]
the induced projective action is
\begin{equation}\label{eq:projAction}
F_{\eps,u}(t)=\frac{(-2+\eps-u)(t-1)-1}{t-1}
+ 1=-1+\eps-u-\frac{1}{t-1}.
\end{equation}
For $\eps=u=0$, the fixed-point equation is
\[
t=-1-\frac{1}{t-1},
\]
giving that the parabolic fixed point is $t=0$.

\smallskip 
\noindent \textbf{(G1).}
The dependence of $A_{\eps,u}$ on $(\eps,u)$ is polynomial, so the induced projective action is real-analytic on $\mathbb{RP}^1.$ Therefore the induced family $G_{\eps,u}$ is $C^1$ in $(\eps,u,x)$.

\smallskip
\noindent \textbf{(G2).}
By the choice of $\varphi$, the unique parabolic fixed point of $g_{0,0}$ is at $0$. Reading off the first two nonzero terms in the Maclaurin expansion of $ F_{0,0}$, we get
\[
\frac{\partial F_{0,0}}{\partial t}(0)=1,
\qquad
\frac{\partial^2 F_{0,0}}{\partial t^2}(0)=2>0.
\]
Passing from this local coordinate to the orientation-preserving coordinate given by $\varphi$ preserves the sign of the quadratic coefficient. Hence
\[
\frac{\partial G_{0,0}(0)}{\partial x}=1
\quad \text{and} \quad
\frac{\partial^2 G_{0,0}(0)}{\partial x^2}>0.
\]
This is the quadratic parabolic case, so there are no intermediate derivatives to check.

\smallskip
\noindent \textbf{(G3)--(G4).}
From \eqref{eq:projAction}, the projective action is increasing in
$\eps$ and decreasing in $u$. Passing to the induced circle map and to our
chosen lift preserves these order relations. In
particular,
\[
\frac{\partial G_{\eps,u}(x)}{\partial \eps}\ge 0,
\qquad
\frac{\partial G_{\eps,u}(x)}{\partial u}\le 0.
\]
Moreover, the local formula above shows that these inequalities are strict at
$(\eps,u,x)=(0,0,0)$.

\smallskip
\noindent \textbf{(M1).}
The pushforward measure $\nu,$ the distribution of $u=V_\omega(n)-a,$ has
compact support $[0,b-a]\subset\mathbb{R}_+$. By \ref{A1}, $a\neq b$,
so $\supp(\nu)$ contains at least two points.

\smallskip
\noindent \textbf{(M2).}
By \ref{A2}, for all sufficiently small $\eps>0$,
\[
\nu([0,\eps])=\mu([a,a+\eps])\ge C_1\eps^{l_1}.
\]

\smallskip
\noindent \textbf{(M3).} Let $x_*$ correspond to the lift of $[0:1]\in\mathbb{RP}^1$. Since $A_{0,0}[0:1]=[1:0]$ and since $A_{0,0}$ has a fixed point, by choice of orientation we have $G_{0,0}(x_*)>x_*.$ Because $A_{\eps,u}[0:1]=[1:0]$ independent of $\eps,u$ and $G_{\eps,u}$ is chosen continuously in $\eps,u$, we have $G_{\eps,u}(x_*)=G_{0,0}(x_*)>x_*.$
\end{proof}

We now prove Theorem~\ref{thm:AndersonLifshitz} under assumptions \ref{A1}--\ref{A3}. Having established Lemma~\ref{lemma:anderson-conditions}, we may apply Theorem~\ref{thm:main} to the family $G_{\eps,u}$, with $\eps$ as the deterministic parameter and $u$ as the random parameter distributed according to $\nu$.

\begin{proof}[Proof of Theorem~\ref{thm:AndersonLifshitz}.]

Let $\rho(E)$ denote the projective rotation number of the Schr\"odinger cocycle, normalized as a rotation number for a chosen lift on $\mathbb{RP}^1$. With this convention the IDS satisfies 
\begin{equation}\label{eq:ids-rotation}
k(E)=1-\rho(E).
\end{equation}
Since $k(E_-)=0,$ \eqref{eq:ids-rotation} gives $\rho(E_-)=1.$ Define the edge-normalized rotation number
\begin{equation}\label{eq:edge-normalized}
    \tilde{\rho}(\eps) := \rho(E_-)-\rho(E_-+\eps).
\end{equation}
 Then
 \begin{equation}
 k(E_-+\eps)= 1-\rho(E_-+\eps) = \rho(E_-)-\rho(E_-+\eps) = \tilde{\rho}(\eps).
 \end{equation}

 By the choice of lift, the rotation number associated with the family $G_{\eps,u},$ as defined in Lemma~\ref{lemma:anderson-conditions}, is precisely this edge-normalized quantity $\tilde{\rho}(\eps)$. 
 Hence Theorem~\ref{thm:main}, applied in the quadratic parabolic case, gives 
 \[
 \lim_{\eps \downarrow 0} \frac{\ln(-\ln \tilde{\rho}(\eps))}{\ln(\eps)}=-\frac12.
 \]
Using $k(E_-+\eps)=\tilde{\rho}(\eps),$ we get
\[
\lim_{E\downarrow E_-}
\frac{\ln(-\ln k(E))}{\ln(E-E_-)}
=
-\frac12.
\]

\medskip

The proof at the upper spectral edge is analogous. One writes $E=E_+-\eps=(b+2)-\eps$ and $u=b-V_\omega(0),$ uses \ref{A3}, and applies the same argument to the corresponding edge-normalized rotation number.
Thus
\[
\lim_{E\uparrow E_+}
\frac{\ln\bigl(-\ln(1-k(E))\bigr)}{\ln(E_+-E)}
=
-\frac12.
\]
\end{proof}

\section*{Acknowledgments}
The authors are grateful to Anton Gorodetski for bringing this problem to their attention and to Victor Kleptsyn for helpful comments and discussion.

\bibliographystyle{abbrv}
\bibliography{ref}

@article{Anderson1958,
  author  = {Anderson, P. W.},
  title   = {Absence of diffusion in certain random lattices},
  journal = {Physical Review},
  volume  = {109},
  number  = {5},
  pages   = {1492--1505},
  year    = {1958}
}

@article {AviDamGor2023,
    AUTHOR = {Avila, Artur and Damanik, David and Gorodetski, Anton},
     TITLE = {The spectrum of {S}chr\"odinger operators with randomly
              perturbed ergodic potentials},
   JOURNAL = {Geom. Funct. Anal.},
  FJOURNAL = {Geometric and Functional Analysis},
    VOLUME = {33},
      YEAR = {2023},
    NUMBER = {2},
     PAGES = {364--375},
      ISSN = {1016-443X,1420-8970},
   MRCLASS = {47B37 (35J10 47A10 47B36 82B44)},
  MRNUMBER = {4578461},
MRREVIEWER = {Jaouad\ Sahbani},
       DOI = {10.1007/s00039-023-00632-z},
       URL = {https://doi-org.ezproxy.rice.edu/10.1007/s00039-023-00632-z},
}

@article {BourgainKlein,
    AUTHOR = {Bourgain, Jean and Klein, Abel},
     TITLE = {Bounds on the density of states for {S}chr\"odinger operators},
   JOURNAL = {Invent. Math.},
  FJOURNAL = {Inventiones Mathematicae},
    VOLUME = {194},
      YEAR = {2013},
    NUMBER = {1},
     PAGES = {41--72},
      ISSN = {0020-9910,1432-1297},
   MRCLASS = {35J10 (35B65)},
  MRNUMBER = {3103255},
MRREVIEWER = {Kenichi\ Ito},
       DOI = {10.1007/s00222-012-0440-1},
       URL = {https://doi-org.ezproxy.rice.edu/10.1007/s00222-012-0440-1},
}

@article{CraSim1983,
  author  = {Craig, Walter and Simon, Barry},
  title   = {Log {H}\"older continuity of the integrated density of states for stochastic {J}acobi matrices},
  journal = {Comm. Math. Phys.},
  volume  = {90},
  number  = {2},
  pages   = {207--218},
  year    = {1983}
}

@book {ESO1,
    AUTHOR = {Damanik, David and Fillman, Jake},
     TITLE = {One-dimensional ergodic {S}chr\"odinger operators---{I}.
              {G}eneral theory},
    SERIES = {Graduate Studies in Mathematics},
    VOLUME = {221},
 PUBLISHER = {American Mathematical Society, Providence, RI},
      YEAR = {2022},
     PAGES = {xv+444},
      ISBN = {978-1-4704-5606-1; [9781470470869]; [9781470470852]},
   MRCLASS = {47-01 (28Axx 35Q41 37Axx 47B36 81Q10 82B44)},
  MRNUMBER = {4567742},
MRREVIEWER = {Christian\ Seifert},
}

@book {ESO2,
    AUTHOR = {Damanik, David and Fillman, Jake},
     TITLE = {One-dimensional ergodic {S}chr\"odinger operators---{II}.
              {S}pecific classes},
    SERIES = {Graduate Studies in Mathematics},
    VOLUME = {249},
 PUBLISHER = {American Mathematical Society, Providence, RI},
      YEAR = {2024},
     PAGES = {vii--xvi and 445--1068},
      ISBN = {978-1-4704-6503-2},
   MRCLASS = {47-01 (35J10 47B36 81Q10 82B44)},
  MRNUMBER = {4840232},
MRREVIEWER = {Christian\ Seifert},
}

@article {DamFilGoh2022,
    AUTHOR = {Damanik, David and Fillman, Jake and Gohlke, Philipp},
     TITLE = {Spectral characteristics of {S}chr\"odinger operators
              generated by product systems},
   JOURNAL = {J. Spectr. Theory},
  FJOURNAL = {Journal of Spectral Theory},
    VOLUME = {12},
      YEAR = {2022},
    NUMBER = {4},
     PAGES = {1659--1718},
      ISSN = {1664-039X,1664-0403},
   MRCLASS = {35J10 (37B10 47B36 52C23 58J51 81Q10)},
  MRNUMBER = {4590016},
       DOI = {10.4171/jst/445},
       URL = {https://doi-org.ezproxy.rice.edu/10.4171/jst/445},
}

@article {DamGor2022,
    AUTHOR = {Damanik, David and Gorodetski, Anton},
     TITLE = {Must the spectrum of a random {S}chr\"odinger operator contain
              an interval?},
   JOURNAL = {Comm. Math. Phys.},
  FJOURNAL = {Communications in Mathematical Physics},
    VOLUME = {393},
      YEAR = {2022},
    NUMBER = {3},
     PAGES = {1583--1613},
      ISSN = {0010-3616,1432-0916},
   MRCLASS = {35J10},
  MRNUMBER = {4453241},
MRREVIEWER = {Felipe\ Ponce-Vanegas},
       DOI = {10.1007/s00220-022-04395-w},
       URL = {https://doi-org.ezproxy.rice.edu/10.1007/s00220-022-04395-w},
}

@article{DelyonSouillard1983,
  author  = {Delyon, F. and Souillard, B.},
  title   = {The rotation number for finite difference operators and its properties},
  journal = {Comm. Math. Phys.},
  volume  = {89},
  number  = {3},
  pages   = {415--426},
  year    = {1983}
}

@article {DonVar1975,
    AUTHOR = {Donsker, M. D. and Varadhan, S. R. S.},
     TITLE = {Asymptotics for the {W}iener sausage},
   JOURNAL = {Comm. Pure Appl. Math.},
  FJOURNAL = {Communications on Pure and Applied Mathematics},
    VOLUME = {28},
      YEAR = {1975},
    NUMBER = {4},
     PAGES = {525--565},
      ISSN = {0010-3640,1097-0312},
   MRCLASS = {60J65},
  MRNUMBER = {397901},
MRREVIEWER = {Priscilla\ Greenwood},
       DOI = {10.1002/cpa.3160280406},
       URL = {https://doi.org/10.1002/cpa.3160280406},
}

@misc{DuaGorKle2025,
      title={The fibered rotation number}, 
      author={Pedro Duarte and Anton Gorodetski and Victor Kleptsyn},
      year={2025},
      eprint={arXiv:2512.00195},
      note={arXiv:2512.00195},
      archivePrefix={arXiv},
      primaryClass={math.DS},
      url={https://arxiv.org/abs/2512.00195}, 
}

@misc{Gourmelon,
      title={Rotation numbers of perturbations of smooth dynamics}, 
      author={Nicolas Gourmelon},
      year={2022},
      eprint={arXiv:2002.06783},
      note={arXiv:2002.06783},
      archivePrefix={arXiv},
      primaryClass={math.DS},
      url={https://arxiv.org/abs/2002.06783}, 
}

@article {GorKle2021,
    AUTHOR = {Gorodetski, Anton and Kleptsyn, Victor},
     TITLE = {Parametric {F}urstenberg theorem on random products of
              {$SL(2,\Bbb R)$} matrices},
   JOURNAL = {Adv. Math.},
  FJOURNAL = {Advances in Mathematics},
    VOLUME = {378},
      YEAR = {2021},
     PAGES = {Paper No. 107522, 81},
      ISSN = {0001-8708,1090-2082},
   MRCLASS = {37D25 (37H15 47B80 60B20 82C44)},
  MRNUMBER = {4184298},
MRREVIEWER = {Zhiming\ Li},
       DOI = {10.1016/j.aim.2020.107522},
       URL = {https://doi.org/10.1016/j.aim.2020.107522},
}

@article {GorKle2024,
    AUTHOR = {Gorodetski, Anton and Kleptsyn, Victor},
     TITLE = {{$\log$}-{H}\"older continuity of the rotation number},
   JOURNAL = {Ergodic Theory Dynam. Systems},
  FJOURNAL = {Ergodic Theory and Dynamical Systems},
    VOLUME = {45},
      YEAR = {2025},
    NUMBER = {12},
     PAGES = {3749--3759},
      ISSN = {0143-3857,1469-4417},
   MRCLASS = {37E45 (37E10 81Q10)},
  MRNUMBER = {4986557},
       DOI = {10.1017/etds.2025.10195},
       URL = {https://doi-org.ezproxy.rice.edu/10.1017/etds.2025.10195},
}

@article {Herman1983,
    AUTHOR = {Herman, Michael-R.},
     TITLE = {Une m\'ethode pour minorer les exposants de {L}yapounov et
              quelques exemples montrant le caract\`ere local d'un
              th\'eor\`eme d'{A}rnol' d{} et de {M}oser sur le tore de
              dimension {$2$}},
   JOURNAL = {Comment. Math. Helv.},
  FJOURNAL = {Commentarii Mathematici Helvetici},
    VOLUME = {58},
      YEAR = {1983},
    NUMBER = {3},
     PAGES = {453--502},
      ISSN = {0010-2571,1420-8946},
   MRCLASS = {58F11 (58F15)},
  MRNUMBER = {727713},
MRREVIEWER = {Dietrich\ Flockerzi},
       DOI = {10.1007/BF02564647},
       URL = {https://doi.org/10.1007/BF02564647},
}

@incollection {Kirsch2008,
    AUTHOR = {Kirsch, Werner},
     TITLE = {An invitation to random {S}chr\"odinger operators},
 BOOKTITLE = {Random {S}chr\"odinger operators},
    SERIES = {Panor. Synth\`eses},
    VOLUME = {25},
     PAGES = {1--119},
      NOTE = {With an appendix by Fr\'ed\'eric Klopp},
 PUBLISHER = {Soc. Math. France, Paris},
      YEAR = {2008},
      ISBN = {978-2-85629-254-9},
   MRCLASS = {82B44 (47B80 60H25 81Q10 82-02)},
  MRNUMBER = {2509110},
MRREVIEWER = {Wolfgang\ K\"onig},
}

@article {KirMar1983,
    AUTHOR = {Kirsch, Werner and Martinelli, Fabio},
     TITLE = {Large deviations and {L}ifshitz singularity of the integrated
              density of states of random {H}amiltonians},
   JOURNAL = {Comm. Math. Phys.},
  FJOURNAL = {Communications in Mathematical Physics},
    VOLUME = {89},
      YEAR = {1983},
    NUMBER = {1},
     PAGES = {27--40},
      ISSN = {0010-3616,1432-0916},
   MRCLASS = {82A57 (60H25 82A31)},
  MRNUMBER = {707770},
       URL = {http://projecteuclid.org/euclid.cmp/1103922589},
}

@incollection {KirMet2007,
    AUTHOR = {Kirsch, Werner and Metzger, Bernd},
     TITLE = {The integrated density of states for random {S}chr\"odinger
              operators},
 BOOKTITLE = {Spectral theory and mathematical physics: a {F}estschrift in
              honor of {B}arry {S}imon's 60th birthday},
    SERIES = {Proc. Sympos. Pure Math.},
    VOLUME = {76, Part 2},
     PAGES = {649--696},
 PUBLISHER = {Amer. Math. Soc., Providence, RI},
      YEAR = {2007},
      ISBN = {978-0-8218-4249-2},
   MRCLASS = {82B44 (35R60 47B80 47N50 81Q10)},
  MRNUMBER = {2307751},
MRREVIEWER = {David\ Damanik},
       DOI = {10.1090/pspum/076.2/2307751},
       URL = {https://doi.org/10.1090/pspum/076.2/2307751},
}

@article {KunSou1980,
    AUTHOR = {Kunz, Herv\'e{} and Souillard, Bernard},
     TITLE = {Sur le spectre des op\'erateurs aux diff\'erences finies
              al\'eatoires},
   JOURNAL = {Comm. Math. Phys.},
  FJOURNAL = {Communications in Mathematical Physics},
    VOLUME = {78},
      YEAR = {1980/81},
    NUMBER = {2},
     PAGES = {201--246},
      ISSN = {0010-3616,1432-0916},
   MRCLASS = {39A12 (60H25 82A05)},
  MRNUMBER = {597748},
       URL = {http://projecteuclid.org/euclid.cmp/1103908590},
}

@incollection{LeP1984,
  author    = {Le Page, {\'E}mile},
  title     = {R\'epartition d'\'etat d'un op\'erateur de {S}chr\"odinger al\'eatoire. {D}istribution empirique des valeurs propres d'une matrice de {J}acobi},
  booktitle = {Probability Measures on Groups, {VII} ({O}berwolfach, 1983)},
  series    = {Lecture Notes in Math.},
  volume    = {1064},
  pages     = {309--367},
  publisher = {Springer},
  address   = {Berlin},
  year      = {1984}
}

@article{Lifshitz1965,
doi = {10.1070/PU1965v007n04ABEH003634},
url = {https://dx.doi.org/10.1070/PU1965v007n04ABEH003634},
year = {1965},
month = {apr},
publisher = {},
volume = {7},
number = {4},
pages = {549},
author = {I M Lifshitz},
title = {Energy spectrum structure and quantum states of disordered condensed systems},
journal = {Soviet Physics Uspekhi},
}

@article {LiLu2008,
    AUTHOR = {Li, Weigu and Lu, Kening},
     TITLE = {Rotation numbers for random dynamical systems on the circle},
   JOURNAL = {Trans. Amer. Math. Soc.},
  FJOURNAL = {Transactions of the American Mathematical Society},
    VOLUME = {360},
      YEAR = {2008},
    NUMBER = {10},
     PAGES = {5509--5528},
      ISSN = {0002-9947,1088-6850},
   MRCLASS = {37E45 (37A30 37E10 37H05)},
  MRNUMBER = {2415083},
MRREVIEWER = {Paulo\ R. C. Ruffino},
       DOI = {10.1090/S0002-9947-08-04619-9},
       URL = {https://doi-org.ezproxy.rice.edu/10.1090/S0002-9947-08-04619-9},
}

@article {Ruelle1985,
    AUTHOR = {Ruelle, David},
     TITLE = {Rotation numbers for diffeomorphisms and flows},
   JOURNAL = {Ann. Inst. H. Poincar\'e{} Phys. Th\'eor.},
  FJOURNAL = {Annales de l'Institut Henri Poincar\'e. Physique Th\'eorique},
    VOLUME = {42},
      YEAR = {1985},
    NUMBER = {1},
     PAGES = {109--115},
      ISSN = {0246-0211},
   MRCLASS = {58F11 (34C35)},
  MRNUMBER = {794367},
MRREVIEWER = {Helmut\ R\"ussmann},
       URL = {http://www.numdam.org/item?id=AIHPB_1985__42_1_109_0},
}

@article {Simon1985,
    AUTHOR = {Simon, Barry},
     TITLE = {Lifschitz tails for the {A}nderson model},
   JOURNAL = {J. Statist. Phys.},
  FJOURNAL = {Journal of Statistical Physics},
    VOLUME = {38},
      YEAR = {1985},
    NUMBER = {1-2},
     PAGES = {65--76},
      ISSN = {0022-4715,1572-9613},
   MRCLASS = {82A57 (60H25)},
  MRNUMBER = {784931},
MRREVIEWER = {M.\ Fukushima},
       DOI = {10.1007/BF01017848},
       URL = {https://doi.org/10.1007/BF01017848},
}

@article {Simon1987,
    AUTHOR = {Simon, Barry},
     TITLE = {Internal {L}ifschitz tails},
   JOURNAL = {J. Statist. Phys.},
  FJOURNAL = {Journal of Statistical Physics},
    VOLUME = {46},
      YEAR = {1987},
    NUMBER = {5-6},
     PAGES = {911--918},
      ISSN = {0022-4715,1572-9613},
   MRCLASS = {82A99 (81C10)},
  MRNUMBER = {893123},
MRREVIEWER = {Eugen\ Belokolos},
       DOI = {10.1007/BF01011147},
       URL = {https://doi.org/10.1007/BF01011147},
}

@article {int01,
    AUTHOR = {G. Mezincescu},
     TITLE = {Internal {L}ifschitz singularities of disordered finite-difference {S}chr\"odinger operators},
   JOURNAL = {Commun. Math. Phys.},
    VOLUME = {103},
      YEAR = {1986},
     PAGES = {167--176},
}

@article {int03,
    AUTHOR = {F. Klopp},
     TITLE = {Internal {L}ifshitz tails for {S}chr\"odinger operators with random potentials},
   JOURNAL = {J. Math. Phys.},
    VOLUME = {43},
    NUMBER = {6},
      YEAR = {2002},
     PAGES = {2948--2958},
}

@article {int04,
    AUTHOR = {Klopp, F. and Wolff, T.},
     TITLE = {Lifshitz tails for 2-dimensional random {S}chr\"odinger operators},
   JOURNAL = {J. Anal. Math.},
    VOLUME = {88},
      YEAR = {2002},
     PAGES = {63--147},
}

@article {LiWu,
    AUTHOR = {Li, Xianzhe and Wu, Li},
     TITLE = {The fibered rotation number for ergodic symplectic cocycles
              and its applications: {I}. {G}ap labelling theorem},
   JOURNAL = {Math. Z.},
  FJOURNAL = {Mathematische Zeitschrift},
    VOLUME = {311},
      YEAR = {2025},
    NUMBER = {3},
     PAGES = {Paper No. 53, 29},
      ISSN = {0025-5874,1432-1823},
   MRCLASS = {37H15 (37E45 37K60 47B36)},
  MRNUMBER = {4958517},
MRREVIEWER = {Anton\ Gorodetski},
       DOI = {10.1007/s00209-025-03860-1},
       URL = {https://doi.org/10.1007/s00209-025-03860-1},
}

@book {Johnson,
    AUTHOR = {Johnson, Russell and Obaya, Rafael and Novo, Sylvia and
              N\'u\~nez, Carmen and Fabbri, Roberta},
     TITLE = {Nonautonomous linear {H}amiltonian systems: oscillation,
              spectral theory and control},
    SERIES = {Developments in Mathematics},
    VOLUME = {36},
 PUBLISHER = {Springer, [Cham]},
      YEAR = {2016},
     PAGES = {xxi+497},
      ISBN = {978-3-319-29023-2; 978-3-319-29025-6},
   MRCLASS = {34-02 (34A30 34B20 37C40 37J99 49N10 70Hxx 93C05)},
  MRNUMBER = {3469433},
MRREVIEWER = {Diogo\ Pinheiro},
       DOI = {10.1007/978-3-319-29025-6},
       URL = {https://doi.org/10.1007/978-3-319-29025-6},
}
\end{document}